\documentclass[twoside]{article}
\usepackage{lipsum}
\usepackage[margin=4cm]{geometry}
\usepackage{amsmath, amsthm, amssymb}
\usepackage{fancyhdr}
\usepackage{enumerate}
\usepackage[english]{babel}
\usepackage[autostyle]{csquotes}
\newcommand{\Z}{\mathbb Z}

 \newtheorem{thm}{Theorem}[section]

 \newtheorem{prop}[thm]{Proposition}
 \theoremstyle{definition}

 \newtheorem{rem}[thm]{Remark}
 \numberwithin{equation}{section}

\DeclareMathOperator{\sym}{sym}

\DeclareMathOperator{\GL}{GL}
\DeclareMathOperator{\SL}{SL}

\fancypagestyle{mypagestyle}{%
  \fancyhf{}
  \fancyhead[OC]{Comparing Hecke eigenvalues of newforms}
  \fancyhead[EC]{\sc{Liubomir Chiriac}}
  \fancyfoot[C]{\thepage}%
}
\title{Comparing Hecke eigenvalues of newforms}
\author{\sc{Liubomir Chiriac}}

\newcommand{\Address}{{
  \bigskip

  \textsc{Liubomir Chiriac} \par\nopagebreak
  \textsc{Department of Mathematics} \par\nopagebreak
  \textsc{University of Massachusetts Amherst} \par\nopagebreak
  \textsc{710 N Pleasant St}, \par\nopagebreak
  \textsc{Amherst, MA 01003}\par\nopagebreak
  \textsc{USA} \par\nopagebreak
  \textit{E-mail address}: \texttt{chiriac@math.umass.edu}

}}

\begin{document}
\date{}
\maketitle

\begin{abstract}
Given two distinct newforms with real Fourier coefficients, we show that the set of primes where the Hecke eigenvalues of one of them dominate the Hecke eigenvalues of the other has density $\geq 1/16$. Furthermore, if the two newforms do not have complex multiplication, and neither is a quadratic twist of the other, we also prove a similar result for the squares of their Hecke eigenvalues. 
\end{abstract}

\section{Introduction}

Let $k\geq 2$ be an even integer and $N$ a positive integer. Denote by $S_k(N)^{\text{new}}$ the set of all cuspidal newforms of weight $k$ for the congruence subgroup $\Gamma_0(N)\subseteq \SL_2(\Z)$ with trivial Nebentypus. Any $f\in S_k(N)^{\text{new}}$ has 
a Fourier expansion at infinity in the upper-half plane $\Im(z)>0$ $$f(z)=\sum_{n\geq 1} a_f(n) e^{2\pi inz}.$$

\noindent The fact that $f$ has trivial Nebetypus implies that the number field $K_f$ generated by all the Fourier coefficients $a_f(n)$ is totally real. We fix a real embedding of $K_f$ and note that our analysis below is independent of this choice. Since $f$ is a newform, it is a simultaneous eigenform for all the Hecke operators, with the corresponding Hecke eigenvalues $a_f(n)$. For convenience, we shall consider the normalized eigenvalues
\[ \lambda_f(n)=\frac{a_f(n)}{n^{(k-1)/2}}.\] 
The extent to which the signs of $\lambda_f(p)$ at primes $p$ determine $f$ uniquely has been first studied by Kowalski, Lau, Soundararajan and Wu \cite{KLSW} (and also by Matom\"aki \cite{Mat}, who refined some of their results). We shall concern ourselves with the following related question: 

\begin{quotation}
Is it possible that for two distinct newforms $f$ and $g$ the eigenvalues $\lambda_f(p)$ are not less than $\lambda_g(p)$, for almost all primes $p$?
\end{quotation}
	
\noindent The first result of this paper is to show that the above situation cannot occur. More precisely, we prove that the set of primes where the eigenvalues of $f$ is strictly less than the corresponding eigenvalues of $g$ has analytic density $\geq 1/16$. Throughout the paper, we shall say that a set $\mathcal{A} $ of primes has analytic density (or Dirichlet density) $\delta>0$ if and only if \[\sum_{p\in \mathcal{A}}\frac{1}{p^s}\sim \delta\sum_{p}\frac{1}{p^s} \text{ as } s\to 1^{+}.\]

\begin{thm} \label{thm1}
Let $k_1,k_2\geq 2$ be even integers and $N_1,N_2\geq 1$ integers. Let $f\in S_{k_1}(N_1)^{\text{new}}$ and $g\in S_{k_2}(N_2)^{\text{new}}$ be two distinct newforms. Then the set $$\{p~|~\lambda_f(p)<\lambda_g(p)\}$$ has analytic density at least $1/16$.
\end{thm} 

\begin{rem} By symmetry, the theorem implies that the set $$\mathcal{F}'=\{p~|~\lambda_f(p)>\lambda_g(p)\}$$ also has density at least $1/16$. Thus, the set $\{p~|~\lambda_f(p)=\lambda_g(p)\}$ where the corresponding Hecke eigenvalues of $f$ and $g$ agree has density at most $7/8$. This is consistent with Ramakrishnan's refinement of the strong multiplicty one for $\GL(2)$ from \cite{DR}, which states that if two cuspidal automorphic representations of $\GL(2)$ over a global field have isomorphic local components \textit{outside} a set of primes of density less than $1/8$, then they are globally isomorphic. 
\end{rem}

\noindent The second result addresses the same question about the squares of the eigenvalues of $f$ and $g$. Assuming that $f$ and $g$ do not have complex multiplication, and that neither is a quadratic twist of each other, we can obtain a similar density estimate.

\begin{thm} \label{thm2}
Let $k_1,k_2\geq 2$ be even integers and $N_1,N_2\geq 1$ integers. Let $f\in S_{k_1}(N_1)^{\text{new}}$ and $g\in S_{k_2}(N_2)^{\text{new}}$ be two newforms without complex multiplication, such that neither is a quadratic twist of the other. Then the set $$\{p~|~\lambda^2_f(p)<\lambda^2_g(p)\}$$ has analytic density at least $1/16$.
\end{thm}

\noindent The strategy of the proof is similar to the one used by Kowalski et al. in \cite{KLSW}, and essentially involves Serre's idea from \cite{Sha} to study the distribution of the Hecke eigenvalues based on the the properties of the first symmetric power $L$-functions. The main ingredients are contained in Section~\ref{inter}. We note that the additional hypotheses in Theorem~\ref{thm2} are required in part (ii) of Proposition~\ref{prop}, where we use the existence  of a cuspidal automorphic representation $\pi$ on $\GL(4)$ (proved by Ramakrishnan in \cite{DR1}) such that $L(\pi,s)=L(f\times g,s)$.

\section{Preliminaries} \label{prelim}

Let $f\in S_k(N)^{\text{new}}$ be a newform. From the standard theory of the Hecke operators, one knows that for all integers $u,v\geq 1$ \[\lambda_f(u)\lambda_f(v)=\sum_{\substack{d|(u,v)\\ (d,N)=1}}\lambda_f\left(\frac{uv}{d^2}\right).\]
For $\Re(s)>1$, the $L$-function $L(f,s)=\sum_{n\geq 1}\frac{\lambda_f(n)}{n^s}$ associated to $f$ can be factored into Euler products
\[\prod_{p}\frac{1}{1-\lambda_f(p)p^{-s}+p^{-2s}}=\prod_p \frac{1}{(1-\alpha_p p^{-s})(1-\beta_pp^{-s})}\] where $\alpha_p$ and $\beta_p$ are the roots of the Hecke polynomial $x^2-\lambda_f(p)x+1$, so that 
\[\alpha_p+\beta_p=\lambda_f(p) \text{ and } \alpha_p\beta_p=1.\] 

\noindent The multiplicative properties of the coefficients imply that for $m\geq 2$ and $p\nmid N$ we have 
\[\lambda_f(p^m)=\lambda_f(p^{m-1})\lambda_f(p)-\lambda_f(p^{m-2})\] and by induction
\[\lambda_f(p^m)=\frac{\alpha_p^{m+1}-\beta_p^{m+1}}{\alpha_p-\beta_p}.\]
One can also write 
\[\lambda_f(p^m)=P_m(\lambda_f(p)),\] where $P_m(x)$ is a polynomial of degree $m$ with real coefficients, such that $P_0(x)=1$, $P_1(x)=x$ and 
\[P_m(x)=xP_{m-1}(x)-P_{m-2}(x)\text{, for $m\geq 2$.}\] Note that \[P_m(x)=U_m\left(\frac{x}{2}\right),\] where $U_m(x)$ is the $m$-th Chebyshev polynomials of the second kind. 

\noindent The $m$-th symmetric power $L$-function attached to $f$ is defined by 
\[L(\sym^m f,s)=\prod_p \prod_{j=0}^m (1-\alpha_p^{m-j}\beta_p^{j}p^{-s})^{-1},\]
and it can also be expressed as a Dirichlet series for $\Re(s)>1$:
\[\sum_{n=1}^{\infty} \frac{\lambda_{\sym^m f}(n)}{n^s}=\prod_p \left(1+\frac{\lambda_{\sym^m f}(p)}{p^s}+\frac{\lambda_{\sym^2 f}(p^2)}{p^{2s}}+\dots \right),\] which implies that 
\[\lambda_{\sym^m f}(p)=\sum_{j=0}^{m}\alpha_p^{m-j}\beta_p^{j}=\lambda_f(p^m).\]

\noindent In particular, 
\[L(\sym^2 f,s)=\prod_p \left(1-\frac{\lambda_f(p^2)}{p^s}+\frac{\lambda_f(p^2)}{p^{2s}}-\frac{1}{p^{3s}} \right)^{-1},\] and 
\[\lambda_{\sym^2 f}(p)=\lambda_f(p^2)=\lambda_f^2(p)-1.\]

\section{An intermediary result} \label{inter}

\noindent In this section we prove a proposition, which will be needed in the proof of Theorem~\ref{thm1} and Theorem~\ref{thm2}. It is based on the holomorphy and non-vanishing at $s=1$ of the symmetric square and symmetric fourth power $L$-functions. In addition, a modularity result of the Rankin-Selberg $L$-series is also used. 

\begin{prop} \label{prop} Let $k_1,k_2\geq 2$ be even integers and $N_1,N_2\geq 1$ integers.
	\begin{enumerate}[(i)] 
	\item Let $f\in S_{k_1}(N_1)^{\text{new}}$ and $g\in S_{k_2}(N_2)^{\text{new}}$ be two distinct newforms. Then \[\sum_{p}\frac{(\lambda_f(p)-\lambda_g(p))^2}{p^s}=2\sum_{p}\frac{1}{p^s}+O(1)\text{, as } s\to 1^{+}.\]
	\item Let $f\in S_{k_1}(N_1)^{\text{new}}$ and $g\in S_{k_2}(N_2)^{\text{new}}$ be two newforms without complex multiplication, such that neither is a quadratic twist of the other. Then \[\sum_{p}\frac{(\lambda_{f}^2(p)-\lambda_{g}^2(p))^2}{p^s}=2\sum_{p}\frac{1}{p^s}+O(1)\text{, as } s\to 1^{+}.\]
	\end{enumerate}
\end{prop}

\begin{proof} It is known that for $m\in \{2,4\}$, there exist cuspidal autmorphic representations $\pi_f^{(m)}$ on $\GL(m+1)$ such that $L(\sym^m f,s)=L(\pi_f^{(m)},s)$.  Indeed, the case $m=2$ is due to Gelbart and Jacquet \cite{GJ}, while the case $m=4$ follows from the work of Kim and Shahidi \cite{KS} on the symmetric fourth power. Recall that the coefficients of the $m$-th symmetric power $L$-functions are given by \[\lambda_{\sym^m f}(p)=\lambda_f(p^m)=P_m(\lambda_f(p)),\] where $P_m(x)$ is the polynomial introduced in Section~\ref{prelim}. Hence, for $m\in\{2,4\}$ we have
\begin{equation}
\sum_{p}\frac{\lambda_f(p^m)}{p^s}=O(1) \text{ as $s\to 1^{+}$}. \label{symsquare}
\end{equation}

\noindent In particular, since $P_4(x)=x^4-3x^2+1$, we obtain from \eqref{symsquare} that  
	\begin{align}
	\sum_p \frac{\lambda_f^4(p)}{p^s}&=\sum_p \frac{\lambda_f(p^4)+3\lambda_f^2(p)-1}{p^s} \notag \\
	&=\sum_p \frac{\lambda_f(p^4)}{p^s} +3\sum_{p} \frac{\lambda_f(p^2)+1}{p^s}-\sum_{p}\frac{1}{p^s} \notag \\
	&=\sum_p \frac{\lambda_f(p^4)}{p^s}+3\sum_p \frac{\lambda_f(p^2)}{p^s}+2\sum_p \frac{1}{p^s} \notag \\
	&=2\sum_p \frac{1}{p^s}+O(1) \text{ as $s\to 1^{+}$},  \label{symfourth}
	\end{align}

\noindent Moreover, since $f$ and $g$ are distinct, the Rankin-Selberg convolution $L(f\times g,s)$ has no poles, and does not vanish at $s=1$. Thus
\begin{equation}
\sum_{p}\frac{\lambda_f(p)\lambda_g(p)}{p^s}=O(1) \text{ as $s\to 1^{+}$}. \label{fnotg}
\end{equation}

\noindent Using \eqref{symsquare} and \eqref{fnotg}, we can now prove part (i): 
\begin{align}
\sum_{p}\frac{(\lambda_f(p)-\lambda_g(p))^2}{p^s}&=\sum_p \frac{\lambda_f^2(p)}{p^s}+\sum_p \frac{\lambda_g^2(p)}{p^s}-2\sum_{p}\frac{\lambda_f(p)\lambda_g(p)}{p^s} \notag \\
&= \sum_p \frac{1+\lambda_f(p^2)}{p^s}+\sum_p \frac{1+\lambda_g(p^2)}{p^s}-2\sum_{p}\frac{\lambda_f(p)\lambda_g(p)}{p^s} \notag \\
&=2\sum_p\frac{1}{p^s}+\sum_p \frac{\lambda_f(p^2)}{p^s}+\sum_p \frac{\lambda_g(p^2)}{p^s}-2\sum_{p}\frac{\lambda_f(p)\lambda_g(p)}{p^s} \notag \\
&=2\sum_p\frac{1}{p^s}+O(1) \text{ as $s\to 1^{+}$}. \notag
\end{align} 

\smallskip

\noindent To prove part (ii) we need an additional ingredient. Namely, if $f$ and $g$ are newforms without complex multiplication, such that neither is a quadratic twist of the other, a result of Ramakrishnan \cite{DR1} guarantees the existence of a cuspidal automorphic representation $\pi$ on $\GL(4)$ such that $L(f\times g,s)=L(\pi,s)$. The coefficients of the Rankin-Selberg $L$-function associated to the pair $(f,g)$ are just $\lambda_f(p)\lambda_g(p)$. Furthermore, $L(\pi\times \tilde \pi,s)$ has a pole of order one at $s=1$, which implies that 
\begin{equation}
\sum_{p}\frac{\lambda_f^2(p)\lambda_g^2(p)}{p^s}=\sum_p\frac{1}{p^s}+O(1) \text{ as $s\to 1^{+}$}.\label{Ram}
\end{equation}

\noindent Part (ii) follows readily from \eqref{symfourth} and \eqref{Ram}:
\begin{align}
\sum_{p}\frac{(\lambda_{f}^2(p)-\lambda_{g}^2(p))^2}{p^s}&=
\sum_p \frac{\lambda_f^4(p)}{p^s}+\sum_p \frac{\lambda_g^4(p)}{p^s}-2\sum_{p}\frac{\lambda_f^2(p)\lambda_g^2(p)}{p^s} \notag \\
&=2\sum_{p}\frac{1}{p^s}+O(1)\text{, as } s\to 1^{+}. \notag
\end{align}
	
\end{proof}

\section{Proof of the main results} The first main result relies on part (i) of Proposition~\ref{prop}. 

\begin{proof}[Proof of Theorem~\ref{thm1}.] For convenience, let us denote $$\mathcal{F}=\{p~|~\lambda_f(p)<\lambda_g(p).\}$$ By Deligne's proof of the Ramanujan-Petersson conjecture we know that for any prime $p$ (not dividing the level of $f$) the following inequality holds
$$|a_f(p)|\leq 2p^{(k_1-1)/2}.$$ Therefore, the normalized coefficients satisfy $|\lambda_f(p)|,|\lambda_g(p)|\leq 2$, so $|\lambda_f(p)-\lambda_g(p)|\leq 4$ and 
\begin{equation} \label{inF}
\sum_{p\in \mathcal{F}}\frac{(\lambda_f(p)-\lambda_g(p))^2}{p^s}\leq 16 \sum_{p\in \mathcal{F}}\frac{1}{p^s}. 
\end{equation}

\noindent Now we estimate the sum over the primes $p$ outside of $\mathcal{F}$ (so $\lambda_f(p)-\lambda_g(p)\geq 0$, whenever $p\notin \mathcal{F}$): 
	\begin{align} \label{notinF}
	\sum_{p\notin \mathcal{F}}\frac{(\lambda_f(p)-\lambda_g(p))^2}{p^s}&
	\leq 4\sum_{p\notin \mathcal{F}}\frac{\lambda_f(p)-\lambda_g(p)}{p^s} \notag \\
	&=4\sum_{p}\frac{\lambda_f(p)-\lambda_g(p)}{p^s}-4\sum_{p\in \mathcal{F}}\frac{\lambda_f(p)-\lambda_g(p)}{p^s}\notag\\ 
	&\leq 16\sum_{p\in \mathcal{F}}\frac{1}{p^s}+O(1) \text{ as $s\to 1^{+}$}. 
	\end{align}
Combing the inequalities \eqref{inF} and \eqref{notinF}, we obtain 
\begin{equation} \label{32}
\sum_{p}\frac{(\lambda_f(p)-\lambda_g(p))^2}{p^s}\leq 32\sum_{p\in \mathcal{F}}\frac{1}{p^s}+O(1) \text{ as $s\to 1^{+}$}.
\end{equation}
Moreover, we know from part (i) of Proposition~\ref{prop} that 
\[\sum_{p}\frac{(\lambda_f(p)-\lambda_g(p))^2}{p^s}=2\sum_{p}\frac{1}{p^s}+O(1) \text{ as $s\to 1^{+}$}.\]
Therefore, as $s\to 1^{+}$ we have that \[\frac{1}{16}\sum_{p}\frac{1}{p^s}+O(1)\leq \sum_{p\in \mathcal{F}}\frac{1}{p^s},\] which shows that the set $\mathcal{F}$ has analytic density at least $1/16$. 
\end{proof}

\noindent The second main result can be obtained similarly, once we use part (ii) of Proposition~\ref{prop}. 

\begin{proof}[Proof of Theorem~\ref{thm2}.] We define the set 
\[\mathcal{S}=\{p~|~\lambda^2_f(p)<\lambda^2_g(p)\}.\] The Deligne bounds still give us $|\lambda_f^2(p)-\lambda_g^2(p)|\leq 4$, and  by the same argument that led to \eqref{32}, one can see that 
$$\sum_{p}\frac{(\lambda_f^2(p)-\lambda_g^2(p))^2}{p^s}\leq 32\sum_{p\in \mathcal{S}}\frac{1}{p^s}+O(1) \text{ as $s\to 1^{+}$}.$$

\noindent Finally, by part (ii) of Proposition~\ref{prop} 
\[\sum_{p}\frac{(\lambda_f^2(p)-\lambda_g^2(p))^2}{p^s}=2\sum_{p}\frac{1}{p^s}+O(1) \text{ as $s\to 1^{+}$},\] and the conclusion follows as before. 
\end{proof}

\subsection*{Acknowledgments} The author thanks Farshid Hajir, Dinakar Ramakrishnan and Siman Wong for helpful conversations.

\Address

\end{document}